\documentclass[a4paper,12pt]{article}
\usepackage[hmargin=2.5cm,vmargin=1.5cm]{geometry}
\usepackage{amsfonts}
\usepackage{amssymb}
\usepackage{amsmath}
\usepackage{amsthm}
\usepackage{amscd}
\usepackage{geometry}
\usepackage{array}
\usepackage{authblk}
\usepackage{float}

\usepackage{pst-node}
\usepackage{tikz-cd}
\usepackage{tikz-qtree}

\usepackage{mathtools}

\newtheorem{theorem}{Theorem}[section]

\newtheorem{remark}{Remark}[section]

\begin{document}

\title{A Banach-Type Fixed Point Theorem via Discrepancy Functionals}

\date{}
\author[]{\small Olaoluwa, Hallowed O.}
\affil[]{\small Department of Mathematics, University of Lagos, Akoka, Nigeria}
\affil[]{\small holaoluwa@unilag.edu.ng}

\maketitle

\begin{abstract}
\noindent
In this paper we establish a Banach-type fixed point theorem for mappings that contract a general nonnegative functional rather than a metric. The principal assumption is that the functional dominates the underlying metric through a simple lower bound, together with a compatibility condition at limits. Our result provides a common framework that recovers recent fixed point theorems in perturbed metric spaces and for polynomial contractions as particular cases, while showing that neither symmetry nor the triangle inequality is essential for the contraction argument. Consequently, the contractive quantity may be interpreted as a general discrepancy or divergence measure rather than a metric. The results highlight that the essential ingredient behind Banach's contraction principle is not the metric structure itself, but the ability of the contracted functional to control the geometry of the underlying space. This perspective provides a simple and flexible framework for obtaining fixed point results for broad classes of non-metrical contractive quantities.
\vspace{5mm}
	
\\

	
\end{abstract}

\section{Introduction}

Since the work of Banach, the contraction mapping principle has become one of the most influential tools in nonlinear analysis. Numerous generalizations have subsequently been developed by modifying either the underlying geometric structure or the form of the contractive condition. Among these are metric-type spaces such as b-metric spaces \cite{B89},\cite{C93}, multiplicative metric spaces \cite{BM},\cite{igez}, generalized metric structures, and many others, together with increasingly sophisticated notions of contractive mappings (see \cite{rhoades},\cite{SUZ}, \cite{WZ}).
\\
\\
A recurring theme throughout this development is that many apparently different fixed point theories are manifestations of a common underlying principle. Identifying this common structure has become an important direction in modern fixed point theory, since unified frameworks not only simplify existing proofs but also clarify the essential assumptions responsible for the existence and uniqueness of fixed points.
\\
\\
The present work continues the author's ongoing research  on unification in fixed point theory. Previous work has sought to demonstrate that several seemingly distinct concepts can be viewed from a common perspective. Examples include the unification of expansive and contractive mappings into hybrid expansive-contractive mappings \cite{inclo}, the interpretation of multipled fixed points as ordinary fixed points in higher-dimensional product spaces \cite{multiplito},\cite{Olaoluwa2015}, and the development of O-metric spaces, which provide a common framework encompassing metric spaces, b-metric spaces, and several related generalized metric structures (see \cite{bookchap},\cite{ourpaper},\cite{ourpaper2}). 
\\
\\
Recently, Jleli and Samet \cite{J25} introduced the notion of perturbed metric spaces, in which the quantity governing the contraction is obtained by perturbing a metric through a nonnegative function. They established a Banach-type fixed point theorem by reducing the problem to the underlying metric. More recently, Jleli, Pacurar and Samet \cite{ETAL25} introduced polynomial contractions, thereby further enlarging the class of contractive expressions. These developments naturally raise the question whether the metric structure of the contracted quantity is really necessary, or is it sufficient that the quantity merely controls the underlying metric.
\\
\\
The purpose of this paper is to answer this question in the affirmative. We replace the perturbed metric by an arbitrary nonnegative functional $A:X \times X \to [0,\infty)$ which is only required to dominate the underlying metric and satisfy a compatibility condition with convergence. No symmetry, triangle inequality, or metric-type axiom is imposed on $A$. Under these assumptions we prove a Banach-type fixed point theorem that contains the recent results on perturbed metric spaces and polynomial contractions as special  cases.

\section{A Fixed Point Theorem via Discrepancy Functionals}

Jleli and Samet \cite{J25} introduced the notion of a perturbed metric as a way of incorporating
measurement errors into the geometry of metric spaces. More precisely, given a nonnegative
perturbation function \(P:X\times X \to [0,\infty)\), a nonnegative function
\(
D:X\times X \to [0,\infty)
\)
is called a \emph{perturbed metric} whenever $D-P$ is a metric on $X$. Their framework defines perturbed convergence, completeness and continuity in terms of the corresponding classical notions in the associated metric space $(X,D-P)$. Their main fixed point theorem may be reformulated as follows.

\begin{theorem}
Let $(X,d)$ be a complete metric space and let
\(
T:X\to X
\)
be a continuous map satisfying
\begin{equation}
d(Tu,Tv)+P(Tu,Tv)
\le
\lambda[d(u,v)+P(u,v)]
\end{equation}
for all $u,v\in X$, where $\lambda\in(0,1)$ and
\(
P:X\times X\to [0,\infty)
\)
is a function. Then $T$ admits a unique fixed point.
\end{theorem}
\noindent
The above theorem shows that fixed point results may still hold even when the
contractive quantity is no longer a metric in the classical sense. A careful analysis
of the proof reveals that the metric structure of the contractive quantity is not
essential. This leads naturally to the following result.

\begin{theorem}[Main Theorem]\label{main}
Let $(X,d)$ be a complete metric space and let
\(
A:X\times X\to [0,\infty)
\)
be a nonnegative function satisfying
\begin{equation}\label{dominance}
A(x,y)\ge \delta d(x,y)^r
\end{equation}
for all $x,y\in X$ and some constants $\delta,r>0$. Suppose that a mapping
\(
T:X\to X
\)
satisfies
\begin{equation}\label{contractivecond}
A(Tx,Ty)\le \lambda A(x,y)
\end{equation}
for all $x,y\in X$, where $\lambda\in(0,1)$. Assume furthermore that for every $u\in X$,
\begin{equation}\label{compatibility}
x_n\to u
\quad \Longrightarrow \quad
A(x_n,u)\to 0.
\end{equation}
\noindent
Then $T$ admits a unique fixed point.
\end{theorem}

\begin{proof}
Without loss of generality, assume $r=1$.  
Fix $x_0\in X$ and define the Picard sequence
\[
x_{n+1}=Tx_n, \quad n \ge 0.
\]
From the contractive condition (\ref{contractivecond}),
\[
A(x_n,x_{n+1})
=
A(Tx_{n-1},Tx_n)
\le
\lambda A(x_{n-1},x_n).
\]
Inductively,
\[
A(x_n,x_{n+1})
\le
\lambda^n A(x_0,x_1).
\]
Given condition (\ref{dominance}), we obtain
\[
d(x_n,x_{n+1})
\le
\frac1\delta A(x_n,x_{n+1})
\le
\frac1\delta \lambda^n A(x_0,x_1).
\]
Now let $m>n$. By the triangle inequality, we have the inequality
\[
d(x_n,x_m)
\le
\sum_{k=n}^{m-1} d(x_k,x_{k+1}).
\]
Hence,
\[
d(x_n,x_m)
\le
\frac{A(x_0,x_1)}{\delta}
\sum_{k=n}^{m-1}\lambda^k,
\]
and since $\lambda\in(0,1)$, the geometric series converges and therefore
$(x_n)$ is a Cauchy sequence. Since $(X,d)$ is complete,
there exists $u\in X$ such that
\[
x_n\to u.
\]
Next,
\[
A(Tx_n,Tu)
\le
\lambda A(x_n,u), \quad n \geq 0.
\]
By assumption (\ref{compatibility}),
\[
A(x_n,u)\to 0,
\]
hence
\[
A(Tx_n,Tu)\to 0.
\]
Using again the domination condition (\ref{dominance}),
\[
d(Tx_n,Tu)
\le
\frac1\delta A(Tx_n,Tu)\to 0,
\]
hence
\[
Tx_n=x_{n+1}\to Tu.
\]
It follows from uniqueness of limits that
\[
Tu=u.
\]
Finally, if $u$ and $v$ are fixed points of $T$, then
\(
A(u,v)
=
A(Tu,Tv)
\le
\lambda A(u,v).
\)
Since $\lambda\in(0,1)$, we have that
\(
A(u,v)=0,
\)
and by condition (\ref{dominance}),
\(
0=A(u,v)\ge \delta d(u,v).
\)
Thus,
\(
d(u,v)=0,
\) 
and $u=v$.
\end{proof}

\begin{remark}
The domination condition (\ref{dominance}) plays a crucial role. The theorem fails when $\delta=0$. Indeed,
let
\(
X=\mathbb{R}
\) 
be endowed with the usual metric 
\(
d(x,y)=|x-y|,
\)
and define
\[
T(x)=x+1.
\]
Clearly, $T$ has no fixed point, yet $T$ contracts the map $A$ defined by
\[
A(x,y)=e^{-(x+y)},
\]
as
\[
A(Tx,Ty)
=
e^{-((x+1)+(y+1))}
=
e^{-2}A(x,y),
\]
Notice that
\[
\frac{A(x,y)}{d(x,y)}
=
\frac{e^{-(x+y)}}{|x-y|}
>0
\]
for all $x\neq y$, but there exists no constant $\delta>0$ such that
\[
A(x,y)\ge \delta d(x,y)
\]
for all $x\neq y$. Thus condition (\ref{dominance}) is essential.
\end{remark}
\noindent
Theorem \ref{main} provides a general framework for the existence of fixed points of maps defined on a complete metric space which contract non-metrical quantities related to the underlying metric in the space by a domination condition, and unifies recent concepts in parallel literature:
\begin{enumerate}
\item If one takes 
\begin{equation}
A(x,y)=d(x,y)+P(x,y), \quad x,y \in X,
\end{equation}
where $(X,d)$ is a complete metric space and $P$ is non-negative, $(X,A,P)$ is a perturbed metric space and Theorem 3.1 in \cite{J25} is retrieved (when $\delta=r=1$).
\item It turns out that Theorem \ref{main} goes beyond perturbed metric spaces. It also generalizes Theorem 2.1 of 
Jleli, Pacurar, and Samet \cite{ETAL25} on \textit{polynomial contractions} - which satisfy our contractive condition (\ref{contractivecond}) in the special case where
\begin{equation}
A(x,y):=\sum_{i=0}^k a_i(x,y)d^i(x,y)
\end{equation}
with $k \geq 1$ a natural number, and $(a_i)_{i=0}^k$ a family of
non-negative valued maps defined on $X \times X$. With their further assumption that  
$$a_j(x,y) \ge A_j, \quad x,y \in X,$$
for some $j \in \{1,2,\ldots,k\}$ and $A_j>0$, the conditions of our Theorem \ref{main} are satisfied.
\item In general, for positive constants $\delta$ and $r$, and a map $g: [0,\infty) \to [0,\infty)$ continuous at $0$ with $g(0)=0$, Theorem \ref{main} holds for
\begin{equation}
A(x,y)=f(d(x,y)),
\end{equation}
where 
\[
f(t)=\delta t^r+g(t).
\] 
\end{enumerate}
\noindent
An important feature of the preceding theorem is that the function $A$ is not required to satisfy the triangle inequality or even symmetry. Thus $A$ may be interpreted more generally as a discrepancy functional or divergence measure. This viewpoint aligns naturally with several modern notions arising in information geometry and optimization theory, which are often not metrics, yet they quantify statistical or informational separation between states.
\\
\\
The theorem above therefore shows that Banach-type fixed point phenomena persist whenever the contracted discrepancy functional dominates the underlying geometry of the space. The inequality (\ref{contractivecond}) may be interpreted as an exponential dissipation of discrepancy under the action of the map $T$, while the domination condition (\ref{dominance}) prevents the discrepancy functional from becoming geometrically degenerate. In particular, small discrepancy necessarily implies
small metric separation. Consequently, the dynamics cannot drift indefinitely while simultaneously decreasing discrepancy, forcing convergence toward a unique equilibrium state.
\\
\\
The significance of the approach in this paper lies in its conceptual simplicity. Rather than introducing yet another generalized metric space, we isolate the property actually used in the proof of Banach's theorem. 
The present work suggests that many recent extensions of Banach's contraction principle may be viewed through discrepancy functionals. Future work will investigate analogous results for multivalued mappings, ordered spaces, and O-metric spaces, as well as applications to optimization and information geometry.

\section*{Authors' Statements}

\subsection*{Conflict of interest}
The author states no conflict of interest.

\subsection*{Availability of data and materials} 
Not Applicable

\subsection*{Funding} 
 None	
%

\subsection*{Acknowledgments}
 None

\end{document}